\documentclass[11pt,reqno,twoside]{amsart}

\usepackage[T1]{fontenc}
\usepackage{lmodern}
\usepackage{amsmath,mathdots,amssymb,latexsym,amsthm,mathrsfs}
\usepackage{dsfont}
\usepackage{xfrac}
\usepackage{breqn}

\usepackage{graphicx}
\usepackage{epsfig}
\usepackage{float}
\usepackage{subfigure}
\usepackage{multirow}
\usepackage{environ}
\usepackage{pifont}
\usepackage{xcolor}
\usepackage{soul}
\usepackage{cite}
\usepackage{microtype}

\usepackage[
  a4paper,
  textwidth=6.7in,
  textheight=9.0in,
  centering
]{geometry}

\renewcommand{\baselinestretch}{1.0}

\allowdisplaybreaks

\AtBeginDocument{%
  \setlength{\abovedisplayskip}{5pt plus 2pt minus 2pt}%
  \setlength{\belowdisplayskip}{5pt plus 2pt minus 2pt}%
  \setlength{\abovedisplayshortskip}{3pt plus 1pt}%
  \setlength{\belowdisplayshortskip}{5pt plus 1pt minus 1pt}%
  \setlength{\jot}{3pt}%
}

\NewEnviron{myequation}{%
  \begin{equation}
    \scalebox{1.1}{$\BODY$}
  \end{equation}
}

\makeatletter

\newcommand{\single}{%
  \let\CS=\@currsize
  \renewcommand{\baselinestretch}{1.1}\tiny\CS
}

\newcommand{\singb}{%
  \let\CS=\@currsize
  \renewcommand{\baselinestretch}{1}\tiny\CS
}

\newcommand{\singa}{%
  \let\CS=\@currsize
  \renewcommand{\baselinestretch}{1.2}\tiny\CS
}

\newcommand{\oneandahalfspacing}{%
  \let\CS=\@currsize
  \renewcommand{\baselinestretch}{1.5}\tiny\CS
}

\newcommand{\singlespacing}{%
  \let\CS=\@currsize
  \renewcommand{\baselinestretch}{1.6}\large\CS
}

\makeatother

\newcommand{\bc}{\begin{center}}
\newcommand{\ec}{\end{center}}

\newcommand{\be}{\begin{eqnarray}}
\newcommand{\ee}{\end{eqnarray}}

\newcommand{\beano}{\begin{eqnarray*}}
\newcommand{\eeano}{\end{eqnarray*}}

\newcommand{\ba}{\begin{array}}
\newcommand{\ea}{\end{array}}

\newcommand{\0}{{_{_{0}}}}

\newcommand{\Hom}{\operatorname{Hom}}
\newcommand{\End}{\operatorname{End}}

\newcommand{\diag}{\operatorname{diag}}

\newcommand{\Ind}{\operatorname{Ind}}

\theoremstyle{plain}

\newtheorem{theorem}{Theorem}[section]
\newtheorem{corollary}[theorem]{Corollary}
\newtheorem{lemma}[theorem]{Lemma}
\newtheorem{prop}[theorem]{Proposition}

\newtheorem{conjecture}[theorem]{Conjecture}
\newtheorem{definition}{Definition}[section]

\theoremstyle{remark}

\numberwithin{equation}{section}

\DeclareMathOperator{\GL}{GL}

\DeclareMathOperator{\G}{G}

\DeclareMathOperator{\F}{F}
\DeclareMathOperator{\N}{N}

\DeclareMathOperator{\V}{V}

\DeclareMathOperator{\B}{B}
\DeclareMathOperator{\To}{T}

\usepackage[final]{hyperref}

\hypersetup{
  colorlinks=true,
  citecolor=red,
  linkcolor=blue,
  urlcolor=blue,
  pdftitle={Parabolically induced representations of GL3 over finite local rings},
  pdfauthor={Prem Dagar},
  pdfkeywords={Parabolic induction; Finite local rings; Principal series representations; Multiplicity-free representations}
}

\begin{document}

\title[Parabolically induced representations of $\GL_n$]
{Parabolically induced representations of $\GL_n(\mathcal O_\ell)$}

\author{Prem Dagar}
\address{Department of Mathematics, Indian Institute of Science Education and Research Tirupati, Tirupati 517619, India}
\email{dagarprem5@gmail.com, prem@labs.iisertirupati.ac.in}



\begin{abstract}
Let $\mathcal O$ be the ring of integers of a non-Archimedean local field with uniformizer $\varpi$, and let $\mathcal O_\ell=\mathcal O/\varpi^\ell\mathcal O$ for $\ell\geq 1$. In this article, we study representations of $\GL_n(\mathcal O_\ell)$ parabolically induced from characters of the Borel subgroup $\B_n(\mathcal O_\ell)$. For $n=3$, we establish a criterion for the irreducibility of these induced representations in terms of the relative conductors of the inducing characters. Motivated by this result, we formulate a conjecture characterizing the irreducibility of such representations for arbitrary $n$. We further show that when all inducing characters are equal, the resulting induced representation is not multiplicity-free, implying that the corresponding triple $(\GL_n(\mathcal O_\ell),\B_n(\mathcal O_\ell),\chi)$ is not a Gelfand pair. Finally, we study the irreducible decomposition of these representations for $\GL_3(\mathcal O_\ell)$ when two relative conductors attain their maximal values. Using the decomposition theory for $\GL_2(\mathcal O_\ell)$, we describe the irreducible constituents and determine their dimensions.
\end{abstract}

\maketitle

\section{Introduction}\label{sec:introduction}

Let $\mathcal O$ be the ring of integers of a non-Archimedean local field $\F$ with maximal ideal $(\varpi)$ and residue field $\mathbb F_q$. For $\ell\geq 1$, let $\mathcal O_\ell=\mathcal O/\varpi^\ell\mathcal O$, which is a finite local chain ring. For $n\geq 2$, write $\G_n=\GL_n(\mathcal O_\ell)$ and $\B_n=\To_n\N_n$, where $\B_n$ is the Borel subgroup of upper triangular matrices, $\To_n$ is the diagonal subgroup, and $\N_n$ is its unipotent radical. Given characters $\chi_1,\ldots,\chi_n$ of $\mathcal O_\ell^\times$, define a character $\chi$ of $\To_n$ by $\chi(\diag(a_1,\ldots,a_n))=\prod_{i=1}^n\chi_i(a_i)$ and extend it trivially to $\B_n$ across $\N_n$. We denote the corresponding induced representation by $\mathcal B(\chi_1,\ldots,\chi_n)=\Ind_{\B_n}^{\G_n}\chi$. In this article, we study these induced representations, focusing on their irreducibility criteria and the number of irreducible constituents in their decompositions.

The representation theory of general linear groups over finite local chain rings has been studied by several authors using congruence subgroups, conjugation orbits, and Clifford theory. Nagornyi~\cite{nag} constructed the irreducible representations of $\GL_2(\mathbb Z/p^\ell\mathbb Z)$ using orbit methods and Clifford theory. Shintani~\cite{shin} introduced the notion of regular representations for $\GL_n(\mathcal O_\ell)$, extending the study to higher rank. Hill~\cite{MR1231615,MR1334228} further developed this approach through conjugation orbits and Clifford theory. Later, Takase~\cite{MR3448171} examined certain aspects of these constructions, while Stasinski and Stevens~\cite{MR3743488} provided a construction of all regular representations of $\GL_n(\mathcal O_\ell)$. Apart from regular representations, those obtained through parabolic induction form another important class, whose irreducibility and decomposition over finite local chain rings remain open questions in general. In a series of articles, Crisp et. al.,~\cite{MR3995719,MR4304676,MR4704476} extend the notion of cusp forms to reductive groups over finite rings. However, the principal-series functor considered in~\cite{MR4304676} need not coincide with parabolic induction from the Borel subgroup. We therefore study the representations $\mathcal B(\chi_1,\ldots,\chi_n)$ directly through their intertwining spaces, with particular emphasis on their irreducibility and decomposition into irreducible constituents. Motivated by this study, we formulate the following conjecture concerning the irreducibility of these representations. We refer to $\S$\ref{pril} for the relevant notation and definitions.

\begin{conjecture}\label{conj:borel-irreducibility} Let $n\geq 2$ and $\ell\geq 2$, and let $\chi_1,\ldots,\chi_n$ be complex characters of $\mathcal O_\ell^\times$. Then the representation $\mathcal B(\chi_1,\ldots,\chi_n)$ is irreducible if and only if the restrictions of $\chi_1,\ldots,\chi_n$ to $1+\varpi^{\ell-1}\mathcal O_\ell$ are pairwise distinct. Equivalently, the conductor $c(\chi_i\chi_j^{-1})=\ell$ for all $1\leq i<j\leq n$. \end{conjecture}

The case $n=2$ of Conjecture~\ref{conj:borel-irreducibility} was proved by the author and Verma~\cite[Theorem~1.1]{premGL2}, who also studied the reducible case. The main result of the present article establishes Conjecture~\ref{conj:borel-irreducibility} for $n=3$, providing a complete characterization of the irreducibility of Borel-induced representations of $\GL_3(\mathcal O_\ell)$ in terms of the relative conductors of the inducing characters. Our proof is based on a direct analysis of the intertwining spaces over the finite local chain ring $\mathcal O_\ell$. Furthermore, the techniques developed here extend beyond the irreducibility problem and allow us to study the decomposition of a class of reducible parabolically induced representations into irreducible constituents. We mention that Campbell and Nevins~\cite{camp1,camp2} studied the restriction of principal series representations of $\GL_3(\F)$ to $\GL_3(\mathcal O)$, where the maximal relative conductor case also arises; see~\cite[Sections~2.1 and~4]{camp2}. In contrast to their restriction-theoretic setting, our approach works directly with the finite group $\GL_3(\mathcal O_\ell)$ and yields both an irreducibility criterion and decomposition results for Borel-induced representations. Extending this analysis to higher rank presents further challenges, as the double-coset structure becomes considerably more intricate for $n>3$ due to the presence of additional congruence parameters within the residue-field Bruhat cells.

The irreducibility criterion naturally leads to a further question concerning the multiplicities of irreducible constituents in the reducible cases. Although the coincidence of two inducing characters guarantees reducibility, the resulting representation may still be multiplicity-free. We first consider the case in which all inducing characters coincide and show that multiplicity-freeness fails for every $n\geq 3$. This gives the following result. 

\begin{theorem}\label{prop:equal-characters-general-degree} 
Let $n\geq 3$, and let $\chi$ be a character of $\mathcal O_\ell^\times$. Then the representation $\mathcal B(\chi,\ldots,\chi)$ is not multiplicity-free. In particular, the triple $(\G_n,\B_n,\chi)$ is not a Gelfand pair. 
\end{theorem} 

The preceding theorem establishes the failure of multiplicity-freeness when all inducing characters coincide. This raises the question of whether multiplicity-free decompositions can still occur under suitable conditions on the inducing characters. In this context, Onn and Singla~\cite{OnnSingla} studied the restriction of unramified principal series representations of $\GL_3(\F)$ to $\GL_3(\mathcal O)$ and gave a complete description of their irreducible constituents. In the present article, we consider the corresponding decomposition problem directly for Borel-induced representations of the finite group $\G_3$. More precisely, we study the case in which the relative conductors $c(\chi_1\chi_3^{-1})$ and $c(\chi_2\chi_3^{-1})$ are maximal, while $c(\chi_1\chi_2^{-1})$ is arbitrary. For this purpose, we consider the standard parabolic subgroup $P$ of type $(2,1)$ defined in~\eqref{eq:parabolic-two-one}. Using the transitivity of parabolic induction together with the known irreducible decomposition of $\mathcal B(\chi_1,\chi_2)$, we obtain an explicit decomposition of $\mathcal B(\chi_1,\chi_2,\chi_3)$ into pairwise inequivalent irreducible constituents. In particular, we establish multiplicity-freeness and determine the exact number of irreducible constituents, as stated in the following theorem.

\begin{theorem}
Let $\chi_1,\chi_2,\chi_3$ be characters of $\mathcal O_\ell^\times$ satisfying $c(\chi_1\chi_3^{-1})=c(\chi_2\chi_3^{-1})=\ell$. Put $m=c(\chi_1\chi_2^{-1})$ and $N=\ell-m+1$. Write $\mathcal B(\chi_1,\chi_2)=V_1\oplus\cdots\oplus V_N$, where $V_1,\ldots,V_N$ are pairwise inequivalent irreducible constituents. Then $$\mathcal B(\chi_1,\chi_2,\chi_3)\cong\bigoplus_{j=1}^{N}\Ind_P^{\G_3}(V_j\otimes\chi_3),$$ where the representations $\Ind_P^{\G_3}(V_j\otimes\chi_3)$ are irreducible and pairwise inequivalent. Consequently, $\mathcal B(\chi_1,\chi_2,\chi_3)$ is multiplicity-free and has exactly $N$ irreducible constituents. 
\end{theorem}

The proofs of our results rely on intertwining-space calculations, Mackey theory, and permutation representations over finite fields. To establish the irreducibility criterion for $\G_3$ in Theorem ~\ref{thm:irreducibility}, we apply Mackey's intertwining formula to compute the endomorphism algebra of $\mathcal B(\chi_1,\chi_2,\chi_3)$. This requires a detailed analysis of the double-coset space $\B_3\backslash\G_3/\B_3$ and the corresponding intertwining conditions on the intersections $\B_3\cap g\B_3g^{-1}$. We show that the identity double coset always contributes a one-dimensional intertwining space, while the contributions from the remaining double cosets are governed by the relative conductors of the inducing characters. In particular, all non-identity contributions vanish precisely when the pairwise relative conductors are maximal, yielding the necessary and sufficient condition for irreducibility and thereby proving Conjecture~\ref{conj:borel-irreducibility} for $n=3$. The proof of Theorem~\ref{prop:equal-characters-general-degree} follows a different approach. When all inducing characters coincide, we reduce the problem, up to twisting by a one-dimensional character, to the representation $\mathcal B(1,\ldots,1)$. Using the permutation representation of $\GL_n(\mathbb F_q)$ on the projective space $\mathbb P^{n-1}(\mathbb F_q)$ and its inflation to $\G_n$, we identify an irreducible constituent of dimension $q+q^2+\cdots+q^{n-1}$ that occurs with multiplicity exactly $n-1$ in $\mathcal B(1,\ldots,1)$. This establishes the failure of multiplicity-freeness for $n\geq 3$. Finally, to prove the decomposition result in Theorem~\ref{thm:separated-character-mf}, we use the transitivity of parabolic induction to express $\mathcal B(\chi_1,\chi_2,\chi_3)$ as a representation induced from the standard parabolic subgroup $P$ of type $(2,1)$. We then use the known irreducible decomposition of $\mathcal B(\chi_1,\chi_2)$ to obtain a direct-sum decomposition into representations induced from its irreducible constituents. The main step is to establish the irreducibility and pairwise inequivalence of these induced representations. For this, we apply Mackey theory to the double-coset space $P\backslash\G_3/P$ and show that the maximal relative-conductor assumptions force the intertwining contributions from all double cosets outside $P$ to vanish. The remaining contribution reduces to the intertwining spaces between the irreducible constituents of $\mathcal B(\chi_1,\chi_2)$. Consequently, each induced summand is irreducible, and distinct summands are inequivalent. This proves that $\mathcal B(\chi_1,\chi_2,\chi_3)$ is multiplicity-free and determines the exact number of its irreducible constituents as $\ell-m+1$, where $m=c(\chi_1\chi_2^{-1})$.

The remainder of the article is organized as follows. In Section~\ref{pril}, we recall the necessary preliminaries, including Mackey theory, irreducibility criteria in terms of endomorphism algebras, relative conductors of characters, and representatives of the relevant double cosets. In Section~\ref{main}, we establish the irreducibility criterion for $\G_3$ and study the failure of multiplicity-freeness when all inducing characters coincide  for $n\geq 3$. Finally, in Section~\ref{sec:separated-principal-series}, we study the irreducible decomposition of $\mathcal B(\chi_1,\chi_2,\chi_3)$ when two relative conductors are maximal. Using the known decomposition of $\mathcal B(\chi_1,\chi_2)$, we obtain an explicit decomposition into irreducible constituents and determine their dimensions.

\section{Preliminaries}\label{pril}

In this section, we recall the necessary preliminaries concerning intertwining spaces, irreducibility and multiplicity-freeness criteria, relative conductors of characters, and essential double-coset representatives. 

\subsection{Intertwining and irreducibility criteria}\label{ssec:mackey-theory}\label{ssec:multiplicity-free}

We begin with the following standard consequence of Mackey's intertwining formula; see~\cite[Theorem~32.1]{bump}.

\begin{lemma}\label{cor1}
Let $G$ be a finite group, and let $H$ and $K$ be subgroups of $G$. Suppose that $\chi_1$ and $\chi_2$ are one-dimensional characters of $H$ and $K$, respectively. For $g\in G$, put $J_g=K\cap gHg^{-1}$ and define $\chi_1^g(h)=\chi_1(g^{-1}hg)$ for $h\in gHg^{-1}$. Then a double coset $KgH$ supports a nonzero intertwining operator between $\Ind_H^G\chi_1$ and $\Ind_K^G\chi_2$ if and only if $\left.\chi_1^g\right|_{J_g}=\left.\chi_2\right|_{J_g}$. Moreover, each such double coset contributes a one-dimensional intertwining space. Consequently, $$\dim\Hom_G(\Ind_H^G\chi_1,\Ind_K^G\chi_2)=\#\{g\in\mathcal D:\left.\chi_1^g\right|_{J_g}=\left.\chi_2\right|_{J_g}\},$$ where $\mathcal D$ is a complete set of representatives for $K\backslash G/H$.
\end{lemma}

To apply this formula to the irreducibility and decomposition of induced representations, we recall the standard endomorphism criterion. A representation is said to be \emph{multiplicity-free} if every irreducible constituent occurs with multiplicity at most one. A triple $(G,H,\phi)$, where $\phi$ is a one-dimensional character of $H$, is called a \emph{Gelfand triple} if $\Ind_H^G\phi$ is multiplicity-free. When $\phi$ is trivial, the pair $(G,H)$ is called a \emph{Gelfand pair}. The following criterion is well known; see~\cite[Theorem~45.1]{bump}.

\begin{prop}\label{th2}
Let $\V$ be a nonzero finite-dimensional complex representation of a finite group $G$. Then $\V$ is multiplicity-free if and only if $\End_G(\V)$ is commutative. In this case, $\dim\End_G(\V)$ equals the number of pairwise inequivalent irreducible constituents of $\V$. Moreover, $\V$ is irreducible if and only if $\dim\End_G(\V)=1$.
\end{prop}

\subsection{Conductors}\label{ssec:conductors}

An additive character of $\mathcal O_\ell$ means a homomorphism
from its additive group to $\mathbb C^\times$, whereas a
multiplicative character means a homomorphism
$\mathcal O_\ell^\times\to\mathbb C^\times$.
Their conductors are defined using the filtrations
\begin{equation*}
\mathcal O_\ell \supseteq \varpi\mathcal O_\ell \supseteq\cdots\supseteq \varpi^{\ell}\mathcal O_\ell =\{0\}
\end{equation*}
and
\begin{equation*}
\mathcal O_\ell^\times =U^0 \supseteq U^1 \supseteq\cdots\supseteq U^{\ell}=\{1\},
\end{equation*}
where $U^r=1+\varpi^r\mathcal O_\ell$ for $1\leq r\leq\ell$.

\begin{definition}\label{cond} 
The conductors of an additive character $\phi$ of $\mathcal O_\ell$ and a multiplicative character $\chi$ of $\mathcal O_\ell^\times$ are defined by $c(\phi)=\min\{0\leq r\leq\ell:\left.\phi\right|_{\varpi^r\mathcal O_\ell}=1\}$ and $c(\chi)=\min\{0\leq r\leq\ell:\left.\chi\right|_{U^r}=1\}$, respectively. In either case, a character is called \emph{primitive} if its conductor is $\ell$.
\end{definition}

For the proof of Conjecture~\ref{conj:borel-irreducibility} when $n=3$, let $\chi_1,\chi_2,\chi_3$ be multiplicative characters of $\mathcal O_\ell^\times$. By Definition~\ref{cond}, for $\ell\geq 2$, we have $$c(\chi_i\chi_j^{-1})=\ell \quad\Longleftrightarrow\quad \left.\chi_i\right|_{U^{\ell-1}}\neq\left.\chi_j\right|_{U^{\ell-1}},$$ for all $1\leq i<j\leq 3$.

\subsection{Double-coset representatives for $\G_3$}\label{ssec:double-cosets}
To apply Lemma~\ref{cor1} in the proof of Theorem~\ref{thm:irreducibility}, we need representatives for the double-coset space $\B_3\backslash\G_3/\B_3$. We recall the representatives given by Onn et al.~\cite[Section~4]{MR2267575} in the following lemma.

\begin{lemma}\label{lem:covering}
Every $\B_3$-double coset in $\G_3$ contains a representative from one of the following families:
\begin{alignat*}{2}
\textup{(a)}\quad g_{i,j,k}
&=\begin{pmatrix}1&0&0\\\varpi^i&1&0\\x\varpi^j&\varpi^k&1\end{pmatrix},
&\qquad\textup{(b)}\quad g_{i,j,1}
&=\begin{pmatrix}0&1&0\\1&0&0\\\varpi^i&\varpi^j&1\end{pmatrix},\\
\textup{(c)}\quad g_{i,j,2}
&=\begin{pmatrix}1&0&0\\\varpi^i&0&1\\\varpi^j&1&0\end{pmatrix},
&\qquad\textup{(d)}\quad g_{i,1}
&=\begin{pmatrix}0&0&1\\1&0&0\\\varpi^i&1&0\end{pmatrix},\\
\textup{(e)}\quad g_{i,2}
&=\begin{pmatrix}0&1&0\\0&\varpi^i&1\\1&0&0\end{pmatrix},
&\qquad\textup{(f)}\quad g
&=\begin{pmatrix}0&0&1\\0&1&0\\1&0&0\end{pmatrix}.
\end{alignat*}
Here $1\leq i,j,k\leq\ell$ in \textup{(a)}, $1\leq i,j\leq\ell$ in \textup{(b)} and \textup{(c)}, and $1\leq i\leq\ell$ in \textup{(d)} and \textup{(e)}, with $x\in\mathcal O_\ell^\times$ in \textup{(a)}. 
\end{lemma}

\section{Parabolically induced representations}\label{main}

In this section, we prove Conjecture~\ref{conj:borel-irreducibility} for $n=3$ and study the multiplicity-freeness of the induced representations when all inducing characters coincide. In particular, we show that the corresponding representations are not multiplicity-free for $n\geq 3$, and hence the associated triples are not Gelfand triples.

\subsection{Irreducibility of $\mathcal B(\chi_1,\chi_2,\chi_3)$}\label{ssec:irreducibility}

Let $\chi_1,\ldots,\chi_n$ be characters of $\mathcal O_\ell^\times$, and let $\chi$ be the character of $\B_n$ defined by $\chi((a_{ij}))=\prod_{i=1}^n\chi_i(a_{ii})$. Recall that $\mathcal B(\chi_1,\ldots,\chi_n)=\Ind_{\B_n}^{\G_n}\chi$. For $n=3$, we have $|\G_3|=q^{9(\ell-1)}(q^3-1)(q^3-q)(q^3-q^2)$ and $|\B_3|=q^{6(\ell-1)+3}(q-1)^3$. Consequently, the dimension of the induced representation is given by $$\dim\mathcal B(\chi_1,\chi_2,\chi_3)=[\G_3:\B_3]=q^{3(\ell-1)}(q+1)(q^2+q+1).$$ 
By Lemma~\ref{cor1} and Proposition~\ref{th2}, the irreducibility of $\mathcal B(\chi_1,\chi_2,\chi_3)$ is determined by the double cosets contributing to its self-intertwining space. We establish Conjecture~\ref{conj:borel-irreducibility} for $n=3$ in the following theorem.

\begin{theorem}\label{thm:irreducibility}

Let $\ell\geq 2$, and let $\chi_1,\chi_2,\chi_3$ be complex characters of $\mathcal O_\ell^\times$. Then the representation $\mathcal B(\chi_1,\chi_2,\chi_3)$ is irreducible if and only if $$\left.\chi_i\right|_{1+\varpi^{\ell-1}\mathcal O_\ell}\neq\left.\chi_j\right|_{1+\varpi^{\ell-1}\mathcal O_\ell}\qquad(1\leq i<j\leq 3).$$
\end{theorem}

\begin{proof}
For $g\in\G_3$, put $H_g=\B_3\cap g\B_3g^{-1}$ and define $\chi^g(y)=\chi(g^{-1}yg)$ for $y\in g\B_3g^{-1}$. Let $\mathcal D$ be a complete set of distinct double-coset representatives obtained from Lemma~\ref{lem:covering}, with $I_3\in\mathcal D$. By Lemma~\ref{cor1}, we have
\begin{equation}\label{eq:mackey}
\End_{\G_3}\bigl(\mathcal B(\chi_1,\chi_2,\chi_3)\bigr)\cong\bigoplus_{g\in\mathcal D}\Hom_{H_g}\bigl(\chi|_{H_g},\chi^g|_{H_g}\bigr).
\end{equation}
Since the inducing characters are one-dimensional, each summand satisfies
\begin{equation}\label{eq:hom}
\dim\Hom_{H_g}\bigl(\chi|_{H_g},\chi^g|_{H_g}\bigr)=\begin{cases}1,&\chi|_{H_g}=\chi^g|_{H_g},\\0,&\chi|_{H_g}\neq\chi^g|_{H_g}.\end{cases}
\end{equation}

\noindent\emph{Sufficiency.} Suppose that the restrictions of $\chi_1,\chi_2,\chi_3$ to $1+\varpi^{\ell-1}\mathcal O_\ell$ are pairwise distinct. Then, for each $1\leq r<s\leq 3$, there exists $t_{rs}\in\varpi^{\ell-1}\mathcal O_\ell$ such that $\chi_r(1+t_{rs})\neq\chi_s(1+t_{rs})$. Since the identity double coset contributes exactly one dimension, by Proposition~\ref{th2}, it suffices to show that every nonidentity double coset contributes trivially to~\eqref{eq:mackey}.

\smallskip
\noindent\emph{Case 1.}
Let
\begin{equation*}
g=
\begin{pmatrix}
0&0&1\\
0&1&0\\
1&0&0
\end{pmatrix}.
\end{equation*}
Since $g^{-1}=g$, one has $H_g=\To_3$. Take
$h=\diag(1+t_{13},1,1)$. Then
\begin{equation*}
g^{-1}hg=\diag(1,1,1+t_{13}).
\end{equation*}
Thus $\chi(h)=\chi_1(1+t_{13})$ and
$\chi^g(h)=\chi_3(1+t_{13})$, so $\chi(h)\neq\chi^g(h)$.

\smallskip
\noindent\emph{Case 2.} Let $g=g_{i,1}=\begin{pmatrix}0&1&0\\0&\varpi^i&1\\1&0&0\end{pmatrix}$, where $1\leq i\leq\ell$. For $X=\begin{pmatrix}a&d&f\\0&b&e\\0&0&c\end{pmatrix}\in\B_3$, a direct computation gives $$gXg^{-1}=\begin{pmatrix}b-e\varpi^i&e&0\\(b-c-e\varpi^i)\varpi^i&c+e\varpi^i&0\\d-f\varpi^i&f&a\end{pmatrix}.$$ Since $H_g=\B_3\cap g\B_3g^{-1}$, the above matrix belongs to $H_g$ if and only if $f=d=0$ and $b-c-e\varpi^i\in\varpi^{\ell-i}\mathcal O_\ell$. Writing $c=b-e\varpi^i+m\varpi^{\ell-i}$, we obtain $$H_g=\left\{\begin{pmatrix}b-e\varpi^i&e&0\\0&b+m\varpi^{\ell-i}&0\\0&0&a\end{pmatrix}:a,b\in\mathcal O_\ell^\times,\ e,m\in\mathcal O_\ell,\ b+m\varpi^{\ell-i}\in\mathcal O_\ell^\times\right\}.$$  For $h\in H_g$ of the above form, we have $$g^{-1}hg=\begin{pmatrix}a&0&0\\0&b&e\\0&0&b-e\varpi^i+m\varpi^{\ell-i}\end{pmatrix},$$ and hence $$\chi(h)=\chi_1(b-e\varpi^i)\chi_2(b+m\varpi^{\ell-i})\chi_3(a),\qquad \chi^g(h)=\chi_1(a)\chi_2(b)\chi_3(b-e\varpi^i+m\varpi^{\ell-i}).$$ By assumption, there exists $t_{13}\in\varpi^{\ell-1}\mathcal O_\ell$ such that $\chi_1(1+t_{13})\neq\chi_3(1+t_{13})$. Taking $a=1+t_{13}$, $b=1$, and $e=m=0$, we obtain $h=\diag(1,1,1+t_{13})\in H_g$, for which $$\chi(h)=\chi_3(1+t_{13})\neq\chi_1(1+t_{13})=\chi^g(h).$$ Thus $\chi|_{H_g}\neq\chi^g|_{H_g}$, and the double coset $\B_3g\B_3$ contributes trivially to~\eqref{eq:mackey}.

\smallskip
\noindent\emph{Case 3.} Let $g=g_{i,2}=\begin{pmatrix}0&0&1\\1&0&0\\\varpi^i&1&0\end{pmatrix}$, where $1\leq i\leq\ell$. For $X=\begin{pmatrix}a&d&f\\0&b&e\\0&0&c\end{pmatrix}\in\B_3$, a direct computation gives $$gXg^{-1}=\begin{pmatrix}c&0&0\\f&a-d\varpi^i&d\\e+f\varpi^i&(a-b-d\varpi^i)\varpi^i&b+d\varpi^i\end{pmatrix}.$$ Since $H_g=\B_3\cap g\B_3g^{-1}$, the above matrix belongs to $H_g$ if and only if $e=f=0$ and $a-b-d\varpi^i\in\varpi^{\ell-i}\mathcal O_\ell$. Writing $b=a-d\varpi^i-m\varpi^{\ell-i}$, we obtain $$H_g=\left\{\begin{pmatrix}c&0&0\\0&a-d\varpi^i&d\\0&0&a-m\varpi^{\ell-i}\end{pmatrix}:a,c\in\mathcal O_\ell^\times,\ d,m\in\mathcal O_\ell,\ a-m\varpi^{\ell-i}\in\mathcal O_\ell^\times\right\}.$$  For $h\in H_g$ of the above form, we have $$g^{-1}hg=\begin{pmatrix}a&d&0\\0&a-d\varpi^i-m\varpi^{\ell-i}&0\\0&0&c\end{pmatrix}.$$ Consequently, $$\chi(h)=\chi_1(c)\chi_2(a-d\varpi^i)\chi_3(a-m\varpi^{\ell-i}),\qquad \chi^g(h)=\chi_1(a)\chi_2(a-d\varpi^i-m\varpi^{\ell-i})\chi_3(c).$$ By assumption, there exists $t_{13}\in\varpi^{\ell-1}\mathcal O_\ell$ such that $\chi_1(1+t_{13})\neq\chi_3(1+t_{13})$. Taking $c=1+t_{13}$, $a=1$, and $d=m=0$, we obtain $h=\diag(1+t_{13},1,1)\in H_g$, for which $g^{-1}hg=\diag(1,1,1+t_{13})$. Hence, we have $$\chi(h)=\chi_1(1+t_{13})\neq\chi_3(1+t_{13})=\chi^g(h).$$ Thus $\chi|_{H_g}\neq\chi^g|_{H_g}$, and the double coset $\B_3g\B_3$ contributes trivially to~\eqref{eq:mackey}.

\smallskip
\noindent\emph{Case 4.} Let $g=g_{i,j,1}=\begin{pmatrix}1&0&0\\\varpi^i&0&1\\\varpi^j&1&0\end{pmatrix}$, where $1\leq i,j\leq\ell$. For $X=\begin{pmatrix}a&d&f\\0&b&e\\0&0&c\end{pmatrix}\in\B_3$, a direct computation gives $$gXg^{-1}=\begin{pmatrix}a-f\varpi^i-d\varpi^j&f&d\\(a-c-f\varpi^i-d\varpi^j)\varpi^i&c+f\varpi^i&d\varpi^i\\(a-b-f\varpi^i-d\varpi^j)\varpi^j-e\varpi^i&e+f\varpi^j&b+d\varpi^j\end{pmatrix}.$$ For this matrix to belong to $H_g$,  we have $(a-c-f\varpi^i-d\varpi^j)\varpi^i=0$, $(a-b-f\varpi^i-d\varpi^j)\varpi^j-e\varpi^i=0$, and $e+f\varpi^j=0$. The last equality gives $e=-f\varpi^j$, and substituting this into the second equality yields $$(a-b-d\varpi^j)\varpi^j=0.$$ Since the annihilator of $\varpi^r$ in $\mathcal O_\ell$ is $\varpi^{\ell-r}\mathcal O_\ell$, we obtain $a-b-d\varpi^j\in\varpi^{\ell-j}\mathcal O_\ell$ and $a-c-f\varpi^i-d\varpi^j\in\varpi^{\ell-i}\mathcal O_\ell$. Hence, we may write $b=a-d\varpi^j-\mu\varpi^{\ell-j}$ and $c=a-f\varpi^i-d\varpi^j-\lambda\varpi^{\ell-i}$ for some $\lambda,\mu\in\mathcal O_\ell$. Consequently, $H_g$ consists of all matrices of the form $$h=\begin{pmatrix}a-f\varpi^i-d\varpi^j&f&d\\0&a-d\varpi^j-\lambda\varpi^{\ell-i}&d\varpi^i\\0&0&a-\mu\varpi^{\ell-j}\end{pmatrix},$$ where $a\in\mathcal O_\ell^\times$, $d,f,\lambda,\mu\in\mathcal O_\ell$, and $a-d\varpi^j-\lambda\varpi^{\ell-i},\,a-\mu\varpi^{\ell-j}\in\mathcal O_\ell^\times$. For such an element $h\in H_g$, we have $$g^{-1}hg=\begin{pmatrix}a&d&f\\0&a-d\varpi^j-\mu\varpi^{\ell-j}&-f\varpi^j\\0&0&a-f\varpi^i-d\varpi^j-\lambda\varpi^{\ell-i}\end{pmatrix}.$$ Therefore, $$\chi(h)=\chi_1(a-f\varpi^i-d\varpi^j)\chi_2(a-d\varpi^j-\lambda\varpi^{\ell-i})\chi_3(a-\mu\varpi^{\ell-j}),$$ whereas $$\chi^g(h)=\chi_1(a)\chi_2(a-d\varpi^j-\mu\varpi^{\ell-j})\chi_3(a-f\varpi^i-d\varpi^j-\lambda\varpi^{\ell-i}).$$ Now choose $t_{23}\in\varpi^{\ell-1}\mathcal O_\ell$ such that $\chi_2(1+t_{23})\neq\chi_3(1+t_{23})$, and take $h=\diag(1,1+t_{23},1)$. Then $g^{-1}hg=\diag(1,1,1+t_{23})$. Thus $h\in H_g$, and $$\chi(h)=\chi_2(1+t_{23})\neq\chi_3(1+t_{23})=\chi^g(h).$$ Hence, we have $\chi|_{H_g}\neq\chi^g|_{H_g}$, and the corresponding double coset contributes trivially to~\eqref{eq:mackey}.

\smallskip
\noindent\emph{Case 5.} Let
$$
g=g_{i,j,2}=
\begin{pmatrix}
0&1&0\\
1&0&0\\
\varpi^i&\varpi^j&1
\end{pmatrix},
\qquad
g^{-1}=
\begin{pmatrix}
0&1&0\\
1&0&0\\
-\varpi^j&-\varpi^i&1
\end{pmatrix},
$$
where $1\leq i,j\leq\ell$. We obtain
$$
gXg^{-1}=
\begin{pmatrix}
b-e\varpi^j&-e\varpi^i&e\\
d-f\varpi^j&a-f\varpi^i&f\\
(b-c-f\varpi^i-e\varpi^j)\varpi^j+d\varpi^i&
(a-c-f\varpi^i-e\varpi^j)\varpi^i&
c+f\varpi^i+e\varpi^j
\end{pmatrix}.
$$
For this matrix to belong to $H_g$, its entries below the diagonal must vanish in $\mathcal O_\ell$. Thus
\begin{align*}
d-f\varpi^j&=0,\\
(b-c-f\varpi^i-e\varpi^j)\varpi^j+d\varpi^i&=0,\\
(a-c-f\varpi^i-e\varpi^j)\varpi^i&=0.
\end{align*}
The first equality gives $d=f\varpi^j$. Substituting this into the second equality yields
$
(b-c-e\varpi^j)\varpi^j=0.
$
Consequently, $b-c-e\varpi^j\in\varpi^{\ell-j}\mathcal O_\ell$ and $a-c-f\varpi^i-e\varpi^j\in\varpi^{\ell-i}\mathcal O_\ell$. Hence, for some $\lambda,\mu\in\mathcal O_\ell$, we may write
$
b=c+e\varpi^j+\lambda\varpi^{\ell-j},
a=c+f\varpi^i+e\varpi^j+\mu\varpi^{\ell-i}.
$
It follows that $H_g$ consists of all matrices of the form
$$
h=
\begin{pmatrix}
c+\lambda\varpi^{\ell-j}&-e\varpi^i&e\\
0&c+e\varpi^j+\mu\varpi^{\ell-i}&f\\
0&0&c+f\varpi^i+e\varpi^j
\end{pmatrix},
$$
where $c\in\mathcal O_\ell^\times$, $e,f,\lambda,\mu\in\mathcal O_\ell$, and
$$
c+\lambda\varpi^{\ell-j},\quad
c+e\varpi^j+\mu\varpi^{\ell-i}
\in\mathcal O_\ell^\times.
$$
For such an element $h\in H_g$, we have
$$
g^{-1}hg=
\begin{pmatrix}
c+f\varpi^i+e\varpi^j+\mu\varpi^{\ell-i}&f\varpi^j&f\\
0&c+e\varpi^j+\lambda\varpi^{\ell-j}&e\\
0&0&c
\end{pmatrix}.
$$
Therefore,
\begin{align*}
\chi(h)
&=\chi_1(c+\lambda\varpi^{\ell-j})
\chi_2(c+e\varpi^j+\mu\varpi^{\ell-i})
\chi_3(c+f\varpi^i+e\varpi^j),\\
\chi^g(h)
&=\chi_1(c+f\varpi^i+e\varpi^j+\mu\varpi^{\ell-i})
\chi_2(c+e\varpi^j+\lambda\varpi^{\ell-j})
\chi_3(c).
\end{align*}
Now choose $t_{12}\in\varpi^{\ell-1}\mathcal O_\ell$ such that $\chi_1(1+t_{12})\neq\chi_2(1+t_{12})$, and take $h=\diag(1+t_{12},1,1)$. Then
$
=\diag(1,1+t_{12},1).
$ Thus $h\in H_g$, and
$$
\chi(h)=\chi_1(1+t_{12})
\neq\chi_2(1+t_{12})=\chi^g(h).
$$
Hence, $\chi|_{H_g}\neq\chi^g|_{H_g}$, and the corresponding double coset contributes trivially to~\eqref{eq:mackey}.

\smallskip
\noindent\emph{Case 6.} Let
\begin{equation*}
g=g_{i,j,k}=
\begin{pmatrix}
1&0&0\\
\varpi^i&1&0\\
x\varpi^j&\varpi^k&1
\end{pmatrix},
\qquad x\in\mathcal O_\ell^\times,
\end{equation*}
where $1\leq i,j,k\leq\ell$. Then
\begin{equation*}
g^{-1}=
\begin{pmatrix}
1&0&0\\
-\varpi^i&1&0\\
\varpi^{i+k}-x\varpi^j&-\varpi^k&1
\end{pmatrix}.
\end{equation*}
For $X=\begin{pmatrix}a&d&f\\0&b&e\\0&0&c\end{pmatrix}\in\B_3$, a direct computation gives
\begin{equation*}
gXg^{-1}=
\begin{pmatrix}
a_{11}&d-f\varpi^k&f\\
a_{21}&a_{22}&e+f\varpi^i\\
a_{31}&a_{32}&a_{33}
\end{pmatrix},
\end{equation*}
where
\begin{align}
a_{11}
&=a-xf\varpi^j+f\varpi^{i+k}-d\varpi^i,\notag\\
a_{22}
&=b-e\varpi^k-f\varpi^{i+k}+d\varpi^i,\notag\\
a_{33}
&=c+xf\varpi^j+e\varpi^k,\notag\\
a_{21}
&=(a-b)\varpi^i-d\varpi^{2i}
+(e+f\varpi^i)(\varpi^{i+k}-x\varpi^j),\notag\\
a_{31}
&=(a-c)x\varpi^j+(c-b)\varpi^{i+k}
-xd\varpi^{i+j}\notag\\
&\qquad +(xf\varpi^j+e\varpi^k)
(\varpi^{i+k}-x\varpi^j),\notag\\
a_{32}
&=(b-c)\varpi^k+xd\varpi^j
-xf\varpi^{j+k}-e\varpi^{2k}.
\label{eq:case-six-entries}
\end{align}
For $h=gXg^{-1}$ to belong to $H_g$, we have $a_{21}=a_{31}=a_{32}=0$. If $i=j=k=\ell$, then $g=I_3$, corresponding to the identity double coset. We therefore assume that $\min\{i,j,k\}\leq\ell-1$. We distinguish the following three subcases according to the relative values of $i,j,k$.

\smallskip
\noindent\emph{Case 6(a): $j\leq i$ and $j\leq k$.}
Choose $t\in\varpi^{\ell-1}\mathcal O_\ell$ such that $\chi_1(1+t)\neq\chi_3(1+t)$. Since $x$ is a unit and $j<\ell$, we can choose $f\in\varpi^{\ell-1-j}\mathcal O_\ell$ satisfying $xf\varpi^j=t$. Take
\begin{equation*}
X=
\begin{pmatrix}
1&0&f\\
0&1&0\\
0&0&1
\end{pmatrix}.
\end{equation*}
Since $i,k\geq j\geq1$, the elements $f\varpi^{i+k}$, $f\varpi^{i+j}$, $f\varpi^{2j}$, and $f\varpi^{j+k}$ all vanish in $\mathcal O_\ell$. It follows from~\eqref{eq:case-six-entries} that $a_{21}=a_{31}=a_{32}=0$, and hence
\begin{equation*}
h=gXg^{-1}=
\begin{pmatrix}
1-t&-f\varpi^k&f\\
0&1&f\varpi^i\\
0&0&1+t
\end{pmatrix}
\in H_g.
\end{equation*}
Since $g^{-1}hg=X$ is upper unitriangular, we have $\chi^g(h)=\chi(X)=1$. On the other hand, since $t^2=0$ in $\mathcal O_\ell$, we have $1-t=(1+t)^{-1}$. Consequently,
\begin{equation*}
\chi(h)=\chi_1(1-t)\chi_3(1+t)
=\frac{\chi_3(1+t)}{\chi_1(1+t)}\neq1.
\end{equation*}
Thus $\chi(h)\neq\chi^g(h)$.

\smallskip
\noindent\emph{Case 6(b): $i<j$ and $i\leq k$.}
Choose $t\in\varpi^{\ell-1}\mathcal O_\ell$ such that $\chi_1(1+t)\neq\chi_2(1+t)$. Since $i<\ell$, there exists $d\in\varpi^{\ell-1-i}\mathcal O_\ell$ satisfying $d\varpi^i=t$. Put
\begin{equation*}
X=
\begin{pmatrix}
1&d&0\\
0&1&0\\
0&0&1
\end{pmatrix}.
\end{equation*}
As $j>i\geq1$, the elements $d\varpi^{2i}$, $d\varpi^j$, and $d\varpi^{i+j}$ vanish in $\mathcal O_\ell$. Therefore, the three lower entries in~\eqref{eq:case-six-entries} are zero, giving
\begin{equation*}
h=gXg^{-1}=
\begin{pmatrix}
1-t&d&0\\
0&1+t&0\\
0&0&1
\end{pmatrix}
\in H_g.
\end{equation*}
Again, $\chi^g(h)=\chi(X)=1$. Since $1-t=(1+t)^{-1}$, we obtain
\begin{equation*}
\chi(h)=\chi_1(1-t)\chi_2(1+t)
=\frac{\chi_2(1+t)}{\chi_1(1+t)}\neq1.
\end{equation*}
Hence, we have $\chi(h)\neq\chi^g(h)$.

\smallskip
\noindent\emph{Case 6(c): $k<i$ and $k<j$.}
Choose $t\in\varpi^{\ell-1}\mathcal O_\ell$ such that $\chi_2(1+t)\neq\chi_3(1+t)$. Since $k<\ell$, choose $e\in\varpi^{\ell-1-k}\mathcal O_\ell$ satisfying $e\varpi^k=t$. Consider
\begin{equation*}
X=
\begin{pmatrix}
1&0&0\\
0&1&e\\
0&0&1
\end{pmatrix}.
\end{equation*}
Since $i,j>k\geq1$, the elements $e\varpi^j$, $e\varpi^{i+k}$, and $e\varpi^{2k}$ vanish in $\mathcal O_\ell$. Thus the three lower entries in~\eqref{eq:case-six-entries} are zero, and we obtain
\begin{equation*}
h=gXg^{-1}=
\begin{pmatrix}
1&0&0\\
0&1-t&e\\
0&0&1+t
\end{pmatrix}
\in H_g.
\end{equation*}
Since $g^{-1}hg=X$, we have $\chi^g(h)=1$, whereas
\begin{equation*}
\chi(h)=\chi_2(1-t)\chi_3(1+t)
=\frac{\chi_3(1+t)}{\chi_2(1+t)}\neq1.
\end{equation*}
Therefore, $\chi(h)\neq\chi^g(h)$.

The above three subcases exhaust all possibilities for $\min\{i,j,k\}<\ell$. In each case, we have constructed an element $h\in H_g$ such that $\chi(h)\neq\chi^g(h)$. Hence, every nonidentity double coset in this family contributes trivially to~\eqref{eq:mackey}. 

Combining Cases 1--6, we conclude that the identity double coset is the only one contributing to the self-intertwining space. Consequently,
\begin{equation*}
\dim\End_{\G_3}\bigl(\mathcal B(\chi_1,\chi_2,\chi_3)\bigr)=1.
\end{equation*}
Thus Proposition~\ref{th2} implies that $\mathcal B(\chi_1,\chi_2,\chi_3)$ is irreducible.

\smallskip
\noindent\emph{Necessity.} Suppose that $\mathcal B(\chi_1,\chi_2,\chi_3)$ is irreducible. Assume, on the contrary, that two of the inducing characters have the same restriction to $1+\varpi^{\ell-1}\mathcal O_\ell$. Put $\delta=\varpi^{\ell-1}$. Since $\ell\geq2$, we have $\delta\neq0$ and $\delta^2=0$ in $\mathcal O_\ell$. We show that there exists a nonidentity double coset for which $\chi|_{H_g}=\chi^g|_{H_g}$, contradicting irreducibility.

\smallskip
\noindent\emph{Case N1.} Suppose that $\chi_1$ and $\chi_2$ have the same restriction to $1+\varpi^{\ell-1}\mathcal O_\ell$. Take $g=I_3+\delta E_{21}$, so that $g^{-1}=I_3-\delta E_{21}$. For
\begin{equation*}
h=
\begin{pmatrix}
a&d&f\\
0&b&e\\
0&0&c
\end{pmatrix}\in\B_3,
\end{equation*}
a direct computation gives
\begin{equation*}
g^{-1}hg=
\begin{pmatrix}
a+\delta d&d&f\\
\delta(b-a)&b-\delta d&e-\delta f\\
0&0&c
\end{pmatrix}.
\end{equation*}
Thus $h\in H_g$ if and only if $\delta(b-a)=0$. For such an element, put $\eta=\delta da^{-1}$. Since $\delta(b-a)=0$ and $a,b$ are units, we also have $\eta=\delta db^{-1}$. Consequently,
\begin{equation*}
\frac{\chi^g(h)}{\chi(h)}
=\chi_1(1+\eta)\chi_2(1-\eta)
=(\chi_1\chi_2^{-1})(1+\eta)=1,
\end{equation*}
where we use $\eta^2=0$, so that $1-\eta=(1+\eta)^{-1}$. The last equality follows from the assumption that $\chi_1$ and $\chi_2$ agree on $1+\varpi^{\ell-1}\mathcal O_\ell$. Hence, $\chi^g(h)=\chi(h)$ for every $h\in H_g$.

\smallskip
\noindent\emph{Case N2.} Suppose that $\chi_1$ and $\chi_3$ have the same restriction to $1+\varpi^{\ell-1}\mathcal O_\ell$. Choose $g=I_3+\delta E_{31}$. Then $g^{-1}=I_3-\delta E_{31}$, and
\begin{equation*}
g^{-1}hg=
\begin{pmatrix}
a+\delta f&d&f\\
\delta e&b&e\\
\delta(c-a)&-\delta d&c-\delta f
\end{pmatrix}.
\end{equation*}
Therefore, $h\in H_g$ if and only if $\delta e=\delta(c-a)=\delta d=0$. Put $\eta=\delta fa^{-1}$. Since $\delta(c-a)=0$, we obtain $\eta=\delta fc^{-1}$. It follows that
\begin{equation*}
\frac{\chi^g(h)}{\chi(h)}
=\chi_1(1+\eta)\chi_3(1-\eta)
=(\chi_1\chi_3^{-1})(1+\eta)=1.
\end{equation*}
Thus $\chi^g(h)=\chi(h)$ for every $h\in H_g$.

\smallskip
\noindent\emph{Case N3.} Suppose that $\chi_2$ and $\chi_3$ have the same restriction to $1+\varpi^{\ell-1}\mathcal O_\ell$. Take $g=I_3+\delta E_{32}$, so that $g^{-1}=I_3-\delta E_{32}$. A direct calculation yields
\begin{equation*}
g^{-1}hg=
\begin{pmatrix}
a&d+\delta f&f\\
0&b+\delta e&e\\
0&\delta(c-b)&c-\delta e
\end{pmatrix}.
\end{equation*}
Hence, $h\in H_g$ if and only if $\delta(c-b)=0$. Let $\eta=\delta eb^{-1}$. The condition $\delta(c-b)=0$ implies $\eta=\delta ec^{-1}$. Therefore,
\begin{equation*}
\frac{\chi^g(h)}{\chi(h)}
=\chi_2(1+\eta)\chi_3(1-\eta)
=(\chi_2\chi_3^{-1})(1+\eta)=1.
\end{equation*}
Thus $\chi^g(h)=\chi(h)$ for every $h\in H_g$.

In each case, $g\notin\B_3$, since $\delta\neq0$. Hence, the corresponding nonidentity double coset contributes an additional dimension to~\eqref{eq:mackey}. Consequently,
\begin{equation*}
\dim\End_{\G_3}\bigl(\mathcal B(\chi_1,\chi_2,\chi_3)\bigr)\geq2.
\end{equation*}
By Proposition~\ref{th2}, the representation $\mathcal B(\chi_1,\chi_2,\chi_3)$ is reducible, contradicting our assumption. Therefore, the restrictions of $\chi_1,\chi_2,\chi_3$ to $1+\varpi^{\ell-1}\mathcal O_\ell$ must be pairwise distinct. This proves necessity and completes the proof.

\end{proof}

\subsection{Equal inducing characters}\label{ssec:equal-characters}

We now study the induced representations when all inducing characters coincide. For $n=3$, the complete irreducible decomposition is known from the work of Onn and Singla~\cite{OnnSingla}. They consider $G=\GL_3(\mathcal O)$, its upper triangular subgroup $B=B(\mathcal O)$, and the representation $V=\Ind_B^G1$, which arises as the restriction of an unramified principal series representation of $\GL_3(\F)$ to $\GL_3(\mathcal O)$. In~\cite[Corollary~6.4]{OnnSingla}, they obtain the complete irreducible decomposition of $V$, with the multiplicities described in~\cite[Proposition~6.3]{OnnSingla} and the dimensions of the irreducible constituents given in~\cite[(6.2)]{OnnSingla}. For $K^\ell=\ker(\GL_3(\mathcal O)\to\GL_3(\mathcal O_\ell))$, the space of $K^\ell$-fixed vectors satisfies $V^{K^\ell}\cong\Ind_{\B_3}^{\G_3}1$. Thus their results also determine the complete decomposition of $\Ind_{\B_3}^{\G_3}1$ for every $\ell\geq1$. Since $\mathcal B(\chi,\chi,\chi)\cong(\chi\circ\det)\otimes\Ind_{\B_3}^{\G_3}1$, the decomposition for arbitrary equal inducing characters follows by twisting. We extend the study of multiplicities to arbitrary $n$ by identifying an explicit irreducible constituent of $\mathcal B(\chi,\ldots,\chi)$ and determining its exact multiplicity. Our approach uses the natural action of $\GL_n(\mathbb F_q)$ on projective space and provides a uniform proof of the failure of multiplicity-freeness for every $n\geq3$.

\begin{proof}[Proof of Theorem~\ref{prop:equal-characters-general-degree}]
Since the character of $\B_n$ determined by $(\chi,\ldots,\chi)$ is the restriction of $\chi\circ\det$ to $\B_n$, the tensor identity for induced representations gives
\begin{equation}\label{eq:general-equal-twist}
\mathcal B(\chi,\ldots,\chi)
\cong(\chi\circ\det)\otimes\Ind_{\B_n}^{\G_n}1.
\end{equation}
Thus it suffices to identify an irreducible constituent of $\Ind_{\B_n}^{\G_n}1$ and determine its multiplicity.
Let $\rho_n:\G_n\to\overline{\G}_n=\GL_n(\mathbb F_q)$ be the reduction homomorphism, and put $\overline{\B}_n=\rho_n(\B_n)$. Consider the projective space $\Omega_n=\mathbb P^{n-1}(\mathbb F_q)$ and the corresponding permutation representation $\mathbb C[\Omega_n]$ of $\overline{\G}_n$. Define
\begin{equation*}
V_n=\left\{f\in\mathbb C[\Omega_n]:\sum_{L\in\Omega_n}f(L)=0\right\}.
\end{equation*}
Then $\mathbb C[\Omega_n]=\mathbb C\mathbf1_{\Omega_n}\oplus V_n$, where $\mathbf1_{\Omega_n}$ denotes the constant function with value one.
The action of $\overline{\G}_n$ on $\Omega_n$ is doubly transitive, since any two ordered pairs of distinct lines can be carried into one another by an invertible linear transformation. Consequently, the diagonal action of $\overline{\G}_n$ on $\Omega_n\times\Omega_n$ has exactly two orbits, namely the diagonal and its complement. It follows that
\begin{equation*}
\dim\End_{\overline{\G}_n}\mathbb C[\Omega_n]=2.
\end{equation*}
Since the action on $\Omega_n$ is transitive, the trivial representation occurs in $\mathbb C[\Omega_n]$ with multiplicity one. By complete reducibility and Schur's lemma, $V_n$ is therefore irreducible. Moreover,
\begin{equation}\label{eq:general-projective-dimension}
\dim V_n=|\Omega_n|-1
=\frac{q^n-1}{q-1}-1
=q+q^2+\cdots+q^{n-1}.
\end{equation}
We next determine the dimension of the space of $\overline{\B}_n$-fixed vectors in $V_n$. For $1\leq r\leq n$, define
\begin{equation*}
\Omega_{n,r}=\left\{\left\langle(a_1,\ldots,a_{r-1},1,0,\ldots,0)^{\mathsf T}\right\rangle:a_1,\ldots,a_{r-1}\in\mathbb F_q\right\}.
\end{equation*}
Every line in $\Omega_n$ has a unique representative whose last nonzero coordinate is $1$. Hence, we have
\begin{equation*}
\Omega_n=\bigsqcup_{r=1}^n\Omega_{n,r},
\qquad
|\Omega_{n,r}|=q^{r-1}.
\end{equation*}
The position of the last nonzero coordinate is preserved by $\overline{\B}_n$. Moreover, the upper triangular unipotent subgroup acts transitively on each $\Omega_{n,r}$. Indeed, for
\begin{equation*}
v=(a_1,\ldots,a_{r-1},1,0,\ldots,0)^{\mathsf T},
\qquad
v'=(a'_1,\ldots,a'_{r-1},1,0,\ldots,0)^{\mathsf T},
\end{equation*}
the matrix
\begin{equation*}
u=I_n+\sum_{i=1}^{r-1}(a'_i-a_i)E_{ir}
\end{equation*}
belongs to $\overline{\B}_n$ and satisfies $uv=v'$. Thus $\Omega_{n,1},\ldots,\Omega_{n,n}$ are precisely the $\overline{\B}_n$-orbits on $\Omega_n$.
It follows that every $\overline{\B}_n$-fixed function on $\Omega_n$ is constant on each $\Omega_{n,r}$ and hence has a unique expression of the form $\sum_{r=1}^na_r\mathbf1_{\Omega_{n,r}}$. Such a function belongs to $V_n$ if and only if
\begin{equation*}
a_1+qa_2+q^2a_3+\cdots+q^{n-1}a_n=0.
\end{equation*}
Since this is one nontrivial linear relation among $n$ variables, we obtain
\begin{equation}\label{eq:general-projective-borel-fixed}
\dim V_n^{\overline{\B}_n}=n-1.
\end{equation}
Now inflate $V_n$ along $\rho_n$ and denote the resulting representation of $\G_n$ by $\widetilde V_n$. Since $\rho_n$ is surjective and $V_n$ is irreducible, $\widetilde V_n$ is also irreducible, with $\dim\widetilde V_n=q+q^2+\cdots+q^{n-1}$. Furthermore, since $\rho_n(\B_n)=\overline{\B}_n$, equation~\eqref{eq:general-projective-borel-fixed} gives
$
\dim\widetilde V_n^{\B_n}
=\dim V_n^{\overline{\B}_n}
=n-1.
$
Applying Frobenius reciprocity, we obtain
\begin{align*}
\dim\Hom_{\G_n}\left(\Ind_{\B_n}^{\G_n}1,\widetilde V_n\right)
&=\dim\Hom_{\B_n}\left(1,\left.\widetilde V_n\right|_{\B_n}\right)=\dim\widetilde V_n^{\B_n}=n-1.
\end{align*}
Since $\G_n$ is finite and $\widetilde V_n$ is irreducible, this dimension is precisely the multiplicity of $\widetilde V_n$ in $\Ind_{\B_n}^{\G_n}1$. Therefore, $\widetilde V_n$ occurs with multiplicity exactly $n-1$.
Finally, by~\eqref{eq:general-equal-twist}, the representation $(\chi\circ\det)\otimes\widetilde V_n$ is an irreducible constituent of $\mathcal B(\chi,\ldots,\chi)$ of dimension $q+q^2+\cdots+q^{n-1}$, occurring with multiplicity exactly $n-1$. In particular, for $n\geq3$, this multiplicity is at least two, so $\mathcal B(\chi,\ldots,\chi)$ is not multiplicity-free. Consequently, the associated Gelfand triple property fails, and taking $\chi=1$ shows that $(\G_n,\B_n)$ is not a Gelfand pair.
\end{proof}

The preceding argument applies uniformly to every $n\geq2$ and $\ell\geq1$. Specializing to $n=3$ gives the following consequence.

\begin{corollary}\label{prop:equal-characters-not-mf}
Let $\ell\geq1$, and let $\chi$ be a character of $\mathcal O_\ell^\times$. Then $\mathcal B(\chi,\chi,\chi)$ has an irreducible constituent of dimension $q^2+q$ occurring with multiplicity exactly two. Consequently, $\mathcal B(\chi,\chi,\chi)$ is not multiplicity-free, and the triple $(\G_3,\B_3,\chi)$ is not a Gelfand triple.
\end{corollary}

\begin{proof}
Taking $n=3$ in the construction given in the proof of Theorem~\ref{prop:equal-characters-general-degree}, we obtain an irreducible constituent of dimension $q^2+q$ with multiplicity $2$, which proves the assertion.
\end{proof}

For $n=2$, the pair $(\G_2,\B_2)$ is a Gelfand pair. The multiplicity-free decomposition of $\mathcal B(\chi,\chi)$ was established by the author and Verma~\cite[Lemma~3.1]{premGL2}. Furthermore, Gupta and Singla~\cite[Theorem 1.1]{GuptaSingla} proved the stronger result that $(\G_2,\B_2)$ is a strong Gelfand pair when the residue characteristic is odd. In contrast, our result shows that for every $n\geq3$, the representation $\mathcal B(\chi,\ldots,\chi)$ is not multiplicity-free. More precisely, the irreducible constituent arising from the projective-space representation has dimension $q+q^2+\cdots+q^{n-1}$ and occurs with multiplicity exactly $n-1$. Thus the multiplicity-free property in degree two does not extend to higher degrees. For $n=3$, this agrees with the complete decomposition obtained by Onn and Singla~\cite[Corollary~6.4]{OnnSingla}, while our projective-space approach establishes the failure of multiplicity-freeness uniformly for all $n\geq3$ and $\ell\geq1$.

\section{Two maximal relative conductors}\label{sec:separated-principal-series}

Assume throughout this section that $c(\chi_1\chi_3^{-1})=c(\chi_2\chi_3^{-1})=\ell$, while $c(\chi_1\chi_2^{-1})$ is arbitrary. We first determine the $n=2$ decomposition needed on the Levi subgroup of the standard parabolic of type $(2,1)$, then use induction through this parabolic to obtain the decomposition for $\G_3$, and finally compute the constituent dimensions and specialize to $\chi_1=\chi_2$.

\subsection{The $n=2$ decomposition}\label{ssec:degree-two}

The required Levi representation is $\mathcal B(\chi_1,\chi_2)$. We use the irreducibility criterion and the endomorphism-algebra calculation established in \cite[Theorems~1.1 and~3.2]{premGL2} to determine its irreducible constituents and their dimensions.

\begin{lemma}\label{lem:rank-two-decomposition}
Let $\chi_1,\chi_2$ be characters of $\mathcal O_\ell^\times$, put $\mathcal B(\chi_1,\chi_2)=\Ind_{\B_2}^{\G_2}(\chi_1\otimes\chi_2)$, and let $m=c(\chi_1\chi_2^{-1})$. Then $\mathcal B(\chi_1,\chi_2)$ is multiplicity-free with exactly $\ell-m+1$ irreducible constituents. If $m\geq1$, their dimensions are $q^{m-1}(q+1)$ and $q^{j-2}(q^2-1)$ for $m<j\leq\ell$. If $m=0$, their dimensions are $1$, $q$, and $q^{j-2}(q^2-1)$ for $2\leq j\leq\ell$.
\end{lemma}

\begin{proof}
Put $\mu=\chi_1\chi_2^{-1}$. By the tensor identity, $\mathcal B(\chi_1,\chi_2)\cong(\chi_2\circ\det)\otimes\mathcal B(\mu,1)$. Since twisting by a one-dimensional character preserves irreducibility, multiplicities, and dimensions, it is enough to consider $\mathcal B(\mu,1)$. By \cite[Theorem~3.2]{premGL2},
\begin{equation}\label{eq:rank-two-end-dimension}
\dim\End_{\G_2}\mathcal B(\mu,1)=\ell-m+1.
\end{equation}

Suppose first that $m\geq1$. For $m\leq j\leq\ell$, the character $\mu$ is trivial on $U^j$ and hence factors through $\mathcal O_j^\times$; let $\mu_j$ denote the resulting character. Let $W_j$ be the inflation to $\G_2$ of $\Ind_{\B_2(\mathcal O_j)}^{\GL_2(\mathcal O_j)}(\mu_j\otimes1)$. Pullback under the reduction maps gives $\G_2$-equivariant inclusions $W_j\subset W_{j+1}$, with $W_\ell=\mathcal B(\mu,1)$, and $\dim W_j=q^{j-1}(q+1)$. Hence,
\begin{equation*}
0\subset W_m\subset W_{m+1}\subset\cdots\subset W_\ell
\end{equation*}
has exactly $\ell-m+1$ nonzero successive quotients.
Since $\G_2$ is finite and the representations are over $\mathbb C$, this filtration splits. Thus there are nonzero $\G_2$-submodules $Q_m,\ldots,Q_\ell$ such that $\mathcal B(\mu,1)\cong Q_m\oplus\cdots\oplus Q_\ell$, where $Q_m\cong W_m$ and $Q_j\cong W_j/W_{j-1}$ for $m<j\leq\ell$. Therefore
\begin{equation*}
\dim\End_{\G_2}\mathcal B(\mu,1)=\sum_{i,j=m}^{\ell}\dim\Hom_{\G_2}(Q_i,Q_j)\geq\ell-m+1.
\end{equation*}
By \eqref{eq:rank-two-end-dimension}, equality holds. Hence, $\dim\End_{\G_2}(Q_j)=1$ for every $j$ and $\Hom_{\G_2}(Q_i,Q_j)=0$ for $i\neq j$. Proposition~\ref{th2} therefore shows that every $Q_j$ is irreducible, while the off-diagonal vanishing shows that they are pairwise inequivalent. Thus $\mathcal B(\mu,1)$ is multiplicity-free with $\ell-m+1$ constituents. Their dimensions are $\dim Q_m=q^{m-1}(q+1)$ and, for $m<j\leq\ell$,
\begin{equation*}
\dim Q_j=\dim W_j-\dim W_{j-1}=q^{j-2}(q^2-1).
\end{equation*}

Suppose now that $m=0$. Then $\mu=1$. For $1\leq j\leq\ell$, let $W_j$ be the inflation to $\G_2$ of $\Ind_{\B_2(\mathcal O_j)}^{\GL_2(\mathcal O_j)}1$. Then $W_1\subset\cdots\subset W_\ell=\mathcal B(1,1)$ and $\dim W_j=q^{j-1}(q+1)$. Since the constant functions form a one-dimensional $\G_2$-subrepresentation of $W_1$, we obtain
\begin{equation*}
0\subset\mathbb C\mathbf1\subset W_1\subset W_2\subset\cdots\subset W_\ell.
\end{equation*}
This filtration has $\ell+1$ nonzero successive quotients, with dimensions $1$, $q$, and $q^{j-2}(q^2-1)$ for $2\leq j\leq\ell$. Since $\dim\End_{\G_2}\mathcal B(1,1)=\ell+1$ by \eqref{eq:rank-two-end-dimension}, the same argument shows that these quotients are irreducible and pairwise inequivalent. Finally, twisting by $\chi_2\circ\det$ preserves dimensions, irreducibility, and multiplicities, and the result follows.
\end{proof}

The preceding lemma gives the complete $n=2$ decomposition needed below. We now show that, under the two maximal relative conductor assumptions, the $\G_3$ intertwining problem reduces entirely to this Levi representation.

\subsection{Induction through the parabolic subgroup}\label{ssec:separated-decomposition}

We now turn to the decomposition of $\mathcal B(\chi_1,\chi_2,\chi_3)$ under the assumptions $c(\chi_1\chi_3^{-1})=c(\chi_2\chi_3^{-1})=\ell$. The main step is to show that all Mackey contributions from double cosets outside the standard parabolic subgroup of type $(2,1)$ vanish. This allows us to reduce the endomorphism-space calculation to the rank-two representation studied in Lemma~\ref{lem:rank-two-decomposition}. For this purpose, let
\begin{equation}\label{eq:parabolic-two-one}
P=\left\{\begin{pmatrix}A&v\\0&c\end{pmatrix}:A\in\G_2,\ v\in\mathcal O_\ell^2,\ c\in\mathcal O_\ell^\times\right\}.
\end{equation}
Write $P=L_PU_P$, where
\begin{equation*}
L_P=\left\{\begin{pmatrix}A&0\\0&c\end{pmatrix}:A\in\G_2,\ c\in\mathcal O_\ell^\times\right\}\simeq\G_2\times\mathcal O_\ell^\times
\end{equation*}
is the Levi subgroup and
\begin{equation*}
U_P=\left\{\begin{pmatrix}I_2&v\\0&1\end{pmatrix}:v\in\mathcal O_\ell^2\right\}
\end{equation*}
is its unipotent radical. For a representation $V$ of $\G_2$, we regard $V\otimes\chi_3$ as a representation of $P$ through the natural projection $P\to L_P$, with $U_P$ acting trivially. With this notation, we now establish Theorem~\ref{thm:separated-character-mf}, which we restate below for convenience.

\begin{theorem}\label{thm:separated-character-mf}
Let $\chi_1,\chi_2,\chi_3$ be characters of $\mathcal O_\ell^\times$ satisfying
\begin{equation}\label{eq:separated-character-hypothesis}
c(\chi_1\chi_3^{-1})=c(\chi_2\chi_3^{-1})=\ell.
\end{equation}
Put $m=c(\chi_1\chi_2^{-1})$ and $N=\ell-m+1$. By Lemma~\ref{lem:rank-two-decomposition}, write $\mathcal B(\chi_1,\chi_2)=V_1\oplus\cdots\oplus V_N$, where $V_1,\ldots,V_N$ are the pairwise inequivalent irreducible constituents of $\mathcal B(\chi_1,\chi_2)$. Then
\begin{equation}\label{eq:separated-parabolic-decomposition}
\mathcal B(\chi_1,\chi_2,\chi_3)\cong\bigoplus_{j=1}^{N}\Ind_P^{\G_3}(V_j\otimes\chi_3),
\end{equation}
and the representations $\Ind_P^{\G_3}(V_j\otimes\chi_3)$ are irreducible and pairwise inequivalent. In particular, $\mathcal B(\chi_1,\chi_2,\chi_3)$ is multiplicity-free with exactly $N$ irreducible constituents.
\end{theorem}

\begin{proof}
Let $\chi$ be the character of $\B_3$ determined by $(\chi_1,\chi_2,\chi_3)$. For $g\in\G_3$, put $H_g=\B_3\cap g\B_3g^{-1}$ and $\chi^g(h)=\chi(g^{-1}hg)$ for $h\in H_g$. Since $\chi|_{H_g}$ and $\chi^g$ are one-dimensional characters, the double coset $\B_3g\B_3$ contributes to the self-intertwining space if and only if $\chi|_{H_g}=\chi^g$. We first prove
\begin{equation}\label{eq:outside-parabolic-vanishing}
\Hom_{H_g}\bigl(\chi|_{H_g},\chi^g\bigr)=0\qquad(g\notin P).
\end{equation}

\smallskip
\noindent\emph{The case $\ell=1$.}
Here $\mathcal O_1=\mathbb F_q$, and the Bruhat decomposition gives permutation matrices as representatives of the $\B_3$-double cosets. The Weyl subgroup corresponding to $P$ is $\{1,(12)\}$. Hence, a permutation matrix $w$ lies outside $P$ if and only if $r=w^{-1}(3)\neq3$, so $r\in\{1,2\}$. By \eqref{eq:separated-character-hypothesis}, $\chi_3\neq\chi_r$. Choose $u\in\mathbb F_q^\times$ such that $\chi_3(u)\neq\chi_r(u)$ and put $h=I_3+(u-1)E_{33}$. Then $h\in\B_3$ and $w^{-1}hw=I_3+(u-1)E_{rr}\in\B_3$, so $h\in H_w$. Moreover, $\chi(h)=\chi_3(u)$ and $\chi^w(h)=\chi_r(u)$, and hence $\chi(h)\neq\chi^w(h)$. This proves \eqref{eq:outside-parabolic-vanishing} for $\ell=1$.
Suppose henceforth that $\ell\geq2$, and put $J=\varpi^{\ell-1}\mathcal O_\ell$. Then $J\varpi\mathcal O_\ell=0$ and $J^2=0$. By \eqref{eq:separated-character-hypothesis}, both $\chi_3\chi_1^{-1}$ and $\chi_3\chi_2^{-1}$ are nontrivial on $1+J$.

We first choose convenient double-coset representatives. Let $g\in\G_3$, and let $\overline g$ denote its reduction modulo $\varpi$. By the Bruhat decomposition of $\overline{\G}_3=\GL_3(\mathbb F_q)$, there are $\overline b_1,\overline b_2\in\overline{\B}_3$ and a permutation matrix $w$ such that $\overline g=\overline b_1w\overline b_2$. Since $\B_3\to\overline{\B}_3$ is surjective, choose lifts $b_1,b_2\in\B_3$ and put $g_0=b_1^{-1}gb_2^{-1}$. Then $g_0$ lies in the same $\B_3$-double coset as $g$ and $g_0\equiv w\pmod{\varpi\mathcal O_\ell}$. Put $M=w^{-1}g_0$. Then $M\equiv I_3\pmod{\varpi\mathcal O_\ell}$. Since the first two pivots are congruent to $1$ modulo $\varpi$, they are units, and Gaussian elimination gives $M=LU$, where $U\in\B_3$ and
\begin{equation*}
L=\begin{pmatrix}1&0&0\\\alpha&1&0\\\beta&\gamma&1\end{pmatrix},\qquad \alpha,\beta,\gamma\in\varpi\mathcal O_\ell.
\end{equation*}
Hence, every $\B_3$-double coset contains a representative
\begin{equation}\label{eq:wL-normal-form}
g=w\begin{pmatrix}1&0&0\\\alpha&1&0\\\beta&\gamma&1\end{pmatrix},\qquad \alpha,\beta,\gamma\in\varpi\mathcal O_\ell.
\end{equation}

\smallskip
\noindent\emph{Permutations moving the third position.}
Suppose $r=w^{-1}(3)\neq3$. Then $r=1$ or $2$. Choose $\varepsilon\in J$ such that $\chi_3(1+\varepsilon)\neq\chi_r(1+\varepsilon)$ and put $h=I_3+\varepsilon E_{33}$. Since $g\equiv w\pmod{\varpi\mathcal O_\ell}$, we have $g^{-1}E_{33}g-E_{rr}\in\varpi M_3(\mathcal O_\ell)$. Multiplication by $\varepsilon\in J$ gives $\varepsilon(g^{-1}E_{33}g-E_{rr})=0$, and hence $g^{-1}hg=I_3+\varepsilon E_{rr}\in\B_3$. Thus $h\in H_g$, while $\chi(h)=\chi_3(1+\varepsilon)\neq\chi_r(1+\varepsilon)=\chi^g(h)$. Therefore, every such double coset has zero intertwining contribution.

It remains to consider the identity permutation and the transposition $(12)$.

\smallskip
\noindent\emph{The transposition $(12)$.}
After left multiplication by a suitable element of $\B_3$, a representative in \eqref{eq:wL-normal-form} may be written as
\begin{equation*}
g=\begin{pmatrix}0&1&0\\1&0&0\\\alpha&\beta&1\end{pmatrix},\qquad g^{-1}=\begin{pmatrix}0&1&0\\1&0&0\\-\beta&-\alpha&1\end{pmatrix},
\end{equation*}
where $\alpha,\beta\in\varpi\mathcal O_\ell$. Here $g\notin P$ if and only if $(\alpha,\beta)\neq(0,0)$.
We use the following elementary fact. If $0\neq z\in\varpi\mathcal O_\ell$, then $z\mathcal O_\ell=\varpi^s\mathcal O_\ell$ for some $1\leq s\leq\ell-1$. Hence, $J=\varpi^{\ell-1}\mathcal O_\ell\subseteq z\mathcal O_\ell$, so for every $\varepsilon\in J$ the equation $vz=\varepsilon$ has a solution $v\in\mathcal O_\ell$.

Suppose first that $\alpha\neq0$. Choose $\varepsilon\in J$ such that $(\chi_3\chi_1^{-1})(1+\varepsilon)\neq1$ and choose $f\in\mathcal O_\ell$ with $f\alpha=\varepsilon$. Put
\begin{equation*}
X=\begin{pmatrix}1+\varepsilon&f\beta&f\\0&1&0\\0&0&1\end{pmatrix}.
\end{equation*}
A direct multiplication gives
\begin{equation*}
gXg^{-1}=\begin{pmatrix}1&0&0\\0&1+\varepsilon-f\alpha&f\\0&\alpha(\varepsilon-f\alpha)&1+f\alpha\end{pmatrix}.
\end{equation*}
Since $f\alpha=\varepsilon$, we obtain
\begin{equation*}
h:=gXg^{-1}=\begin{pmatrix}1&0&0\\0&1&f\\0&0&1+\varepsilon\end{pmatrix}\in\B_3.
\end{equation*}
Thus $h\in H_g$ and $g^{-1}hg=X$. Hence, $\chi(h)=\chi_3(1+\varepsilon)$, whereas $\chi^g(h)=\chi(X)=\chi_1(1+\varepsilon)$, so $\chi(h)\neq\chi^g(h)$.

Suppose next that $\alpha=0$ and $\beta\neq0$. Choose $\varepsilon\in J$ such that $(\chi_3\chi_2^{-1})(1+\varepsilon)\neq1$ and choose $e\in\mathcal O_\ell$ with $e\beta=\varepsilon$. Put
\begin{equation*}
X=\begin{pmatrix}1&0&0\\0&1+\varepsilon&e\\0&0&1\end{pmatrix}.
\end{equation*}
Then
\begin{equation*}
gXg^{-1}=\begin{pmatrix}1+\varepsilon-e\beta&0&e\\0&1&0\\\beta(\varepsilon-e\beta)&0&1+e\beta\end{pmatrix},
\end{equation*}
and hence
\begin{equation*}
h:=gXg^{-1}=\begin{pmatrix}1&0&e\\0&1&0\\0&0&1+\varepsilon\end{pmatrix}\in\B_3.
\end{equation*}
Thus $h\in H_g$ and $g^{-1}hg=X$. Hence, $\chi(h)=\chi_3(1+\varepsilon)$, while $\chi^g(h)=\chi(X)=\chi_2(1+\varepsilon)$, so again $\chi(h)\neq\chi^g(h)$.

\smallskip
\noindent\emph{The identity permutation.}
Now take
\begin{equation*}
g=\begin{pmatrix}1&0&0\\\alpha&1&0\\\beta&\gamma&1\end{pmatrix},\qquad g^{-1}=\begin{pmatrix}1&0&0\\-\alpha&1&0\\\alpha\gamma-\beta&-\gamma&1\end{pmatrix},
\end{equation*}
where $\alpha,\beta,\gamma\in\varpi\mathcal O_\ell$. Here $g\notin P$ if and only if $(\beta,\gamma)\neq(0,0)$.

Suppose first that $\beta\neq0$. Choose $\varepsilon\in J$ such that $(\chi_3\chi_1^{-1})(1+\varepsilon)\neq1$ and choose $f\in\mathcal O_\ell$ with $f\beta=\varepsilon$. Take
\begin{equation*}
X=\begin{pmatrix}1&f\gamma&f\\0&1&0\\0&0&1\end{pmatrix}.
\end{equation*}
A direct multiplication gives
\begin{equation*}
gXg^{-1}=\begin{pmatrix}1-f\beta&0&f\\-\alpha f\beta&1&\alpha f\\-\beta^2f&0&1+f\beta\end{pmatrix}.
\end{equation*}
Since $f\beta=\varepsilon$ and $\alpha,\beta\in\varpi\mathcal O_\ell$, we have $\alpha\varepsilon=\beta\varepsilon=0$. Therefore
\begin{equation*}
h:=gXg^{-1}=\begin{pmatrix}1-\varepsilon&0&f\\0&1&\alpha f\\0&0&1+\varepsilon\end{pmatrix}\in\B_3.
\end{equation*}
Thus $h\in H_g$ and $g^{-1}hg=X$. Since $X$ is unipotent, $\chi^g(h)=1$. Moreover, $\varepsilon^2=0$, so $1-\varepsilon=(1+\varepsilon)^{-1}$, and hence $\chi(h)=(\chi_3\chi_1^{-1})(1+\varepsilon)\neq1$. Thus $\chi(h)\neq\chi^g(h)$.

Finally, suppose that $\beta=0$ and $\gamma\neq0$. Choose $\varepsilon\in J$ such that $(\chi_3\chi_2^{-1})(1+\varepsilon)\neq1$ and choose $e\in\mathcal O_\ell$ with $e\gamma=\varepsilon$. Take
\begin{equation*}
X=\begin{pmatrix}1&0&0\\0&1&e\\0&0&1\end{pmatrix}.
\end{equation*}
A direct multiplication gives
\begin{equation*}
gXg^{-1}=\begin{pmatrix}1&0&0\\\alpha e\gamma&1-e\gamma&e\\\alpha e\gamma^2&-e\gamma^2&1+e\gamma\end{pmatrix}.
\end{equation*}
Since $e\gamma=\varepsilon$ and $\alpha,\gamma\in\varpi\mathcal O_\ell$, we have $\alpha\varepsilon=\gamma\varepsilon=0$. Hence, we have
\begin{equation*}
h:=gXg^{-1}=\begin{pmatrix}1&0&0\\0&1-\varepsilon&e\\0&0&1+\varepsilon\end{pmatrix}\in\B_3.
\end{equation*}
Thus $h\in H_g$ and $g^{-1}hg=X$. Again $\chi^g(h)=1$, whereas $\chi(h)=(\chi_3\chi_2^{-1})(1+\varepsilon)\neq1$. Hence, we have $\chi(h)\neq\chi^g(h)$.
The above cases exhaust the representatives in \eqref{eq:wL-normal-form}. Therefore, \eqref{eq:outside-parabolic-vanishing} holds.

\smallskip
\noindent\emph{Reduction to the Levi subgroup.}
Since $\B_3\subset P$, a $\B_3$-double coset which meets $P$ is contained in $P$. Hence, the double cosets which can contribute are precisely those in $\B_3\backslash P/\B_3$. Comparing the Mackey decompositions in $\G_3$ and in $P$ gives
\begin{equation}\label{eq:separated-end-dimension-first}
\dim\End_{\G_3}\mathcal B(\chi_1,\chi_2,\chi_3)=\dim\End_P\bigl(\Ind_{\B_3}^{P}\chi\bigr).
\end{equation}

Since $P=L_PU_P$, $U_P\subset\B_3$, and $\chi$ is trivial on $U_P$, the representation $\Ind_{\B_3}^{P}\chi$ is obtained from $\mathcal B(\chi_1,\chi_2)\otimes\chi_3$ on $L_P\simeq\G_2\times\mathcal O_\ell^\times$ by letting $U_P$ act trivially. Consequently,

\begin{equation}\label{eq:separated-end-dimension}
\begin{aligned}
\dim\End_{\G_3}\mathcal B(\chi_1,\chi_2,\chi_3)
&=\dim\End_{\G_2\times\mathcal O_\ell^\times}\bigl(\mathcal B(\chi_1,\chi_2)\otimes\chi_3\bigr)\\
&=\dim\End_{\G_2}\mathcal B(\chi_1,\chi_2)\\
&=N.
\end{aligned}
\end{equation}

The second equality holds because the $\mathcal O_\ell^\times$-factor acts through the scalar character $\chi_3$, and the last equality follows from Lemma~\ref{lem:rank-two-decomposition}.
By induction in stages, $\mathcal B(\chi_1,\chi_2,\chi_3)\cong\Ind_P^{\G_3}(\Ind_{\B_3}^{P}\chi)$. Using $\mathcal B(\chi_1,\chi_2)=\bigoplus_{j=1}^{N}V_j$ and additivity of induction gives \eqref{eq:separated-parabolic-decomposition}. Put $W_j=\Ind_P^{\G_3}(V_j\otimes\chi_3)$. Each $W_j$ is nonzero, since $\dim W_j=[\G_3:P]\dim V_j>0$, and $\mathcal B(\chi_1,\chi_2,\chi_3)\cong\bigoplus_{j=1}^{N}W_j$. Therefore
\begin{equation*}
\dim\End_{\G_3}\mathcal B(\chi_1,\chi_2,\chi_3)=\sum_{i,j=1}^{N}\dim\Hom_{\G_3}(W_i,W_j).
\end{equation*}
Each diagonal term has dimension at least one, so the right-hand side is at least $N$. By \eqref{eq:separated-end-dimension}, it is exactly $N$. Hence, $\dim\End_{\G_3}(W_j)=1$ for every $j$ and $\Hom_{\G_3}(W_i,W_j)=0$ for $i\neq j$. Proposition~\ref{th2} shows that every $W_j$ is irreducible, while the off-diagonal vanishing shows that the $W_j$ are pairwise inequivalent. This proves the theorem.
\end{proof}

The theorem reduces the decomposition of $\mathcal B(\chi_1,\chi_2,\chi_3)$ to the $n=2$ Levi representation. We now compute the dimensions of the induced constituents and then specialize to the case $\chi_1=\chi_2$.

\subsection{Constituent dimensions and two equal characters}\label{ssec:constituent-dimensions}

\begin{corollary}\label{cor:separated-constituent-dimensions}
Under the hypotheses of Theorem~\ref{thm:separated-character-mf}, put $D_\ell=[\G_3:P]=q^{2(\ell-1)}(q^2+q+1)$. If $m\geq1$, the constituent dimensions of $\mathcal B(\chi_1,\chi_2,\chi_3)$ are $D_\ell q^{m-1}(q+1)$ and $D_\ell q^{j-2}(q^2-1)$ for $m<j\leq\ell$. If $m=0$, equivalently $\chi_1=\chi_2$, they are $D_\ell$, $qD_\ell$, and $D_\ell q^{j-2}(q^2-1)$ for $2\leq j\leq\ell$. Each constituent occurs with multiplicity one.
\end{corollary}

\begin{proof}
Let $\overline P$ be the image of $P$ in $\overline{\G}_3=\GL_3(\mathbb F_q)$. It is the stabilizer of $\langle e_1,e_2\rangle$, so $[\overline{\G}_3:\overline P]=q^2+q+1$. The kernels of the reduction maps $\G_3\to\overline{\G}_3$ and $P\to\overline P$ have orders $q^{9(\ell-1)}$ and $q^{7(\ell-1)}$, respectively. Hence, we have
\begin{equation*}
[\G_3:P]=q^{2(\ell-1)}(q^2+q+1)=D_\ell.
\end{equation*}
By Theorem~\ref{thm:separated-character-mf}, the constituents are $\Ind_P^{\G_3}(V_j\otimes\chi_3)$, and induction from $P$ multiplies dimensions by $D_\ell$. The stated dimensions therefore follow from Lemma~\ref{lem:rank-two-decomposition}. When $m=0$, Definition~\ref{cond} gives $c(\chi_1\chi_2^{-1})=0$ if and only if $\chi_1\chi_2^{-1}=1$, equivalently $\chi_1=\chi_2$. Multiplicity one follows from Theorem~\ref{thm:separated-character-mf}.
\end{proof}

\begin{corollary}\label{cor:two-equal-separated-characters}
Let $\chi,\eta$ be characters of $\mathcal O_\ell^\times$ satisfying $c(\chi\eta^{-1})=\ell$. Then $\mathcal B(\chi,\chi,\eta)$ is reducible and multiplicity-free, with exactly $\ell+1$ irreducible constituents. In particular, $(\G_3,\B_3,\chi)$ is a Gelfand triple, where
\begin{equation*}
\chi\begin{pmatrix}a&d&f\\0&b&e\\0&0&c\end{pmatrix}=\chi(a)\chi(b)\eta(c).
\end{equation*}
The constituent dimensions are $D_\ell$, $qD_\ell$, and $D_\ell q^{j-2}(q^2-1)$ for $2\leq j\leq\ell$, where $D_\ell=q^{2(\ell-1)}(q^2+q+1)$.
\end{corollary}

\begin{proof}
Take $\chi_1=\chi_2=\chi$ and $\chi_3=\eta$ in Theorem~\ref{thm:separated-character-mf}. Then $c(\chi_1\chi_3^{-1})=c(\chi_2\chi_3^{-1})=\ell$ and $m=c(\chi_1\chi_2^{-1})=0$, so $N=\ell+1$. Hence, $\mathcal B(\chi,\chi,\eta)$ is multiplicity-free with $\ell+1$ irreducible constituents. Since $\ell\geq1$, it has at least two constituents and is therefore reducible. The Gelfand-triple assertion follows from multiplicity-freeness, and the dimensions follow from Corollary~\ref{cor:separated-constituent-dimensions}.
\end{proof}

\subsection*{Acknowledgments} The author thanks Hariom Sharma for carefully reading the article and providing valuable comments. AI tools were used solely for language editing and improving the clarity of the manuscript.

\bibliographystyle{amsplain}
\bibliography{psr}
\end{document}